\documentclass[reqno,oneside,letterpaper]{amsart}
\usepackage{mathrsfs}
\usepackage{amssymb}
\usepackage{wasysym}

\usepackage[utf8]{inputenc}
\usepackage{amsmath, amssymb,bm, cases, mathtools, thmtools}
\usepackage{verbatim}
\usepackage{graphicx}\graphicspath{{figures/}}
\usepackage{multicol}
\usepackage{tabularx}
\usepackage[usenames,dvipsnames]{xcolor}
\usepackage{mathrsfs}
\usepackage{url}
\usepackage[normalem]{ulem}
\usepackage{xstring}
\usepackage{enumitem}

\usepackage[%
    minnames=1,maxnames=99,maxcitenames=3,
    style=alphabetic,
    doi=false,url=false,
    firstinits=true,hyperref,natbib,backend=bibtex]{biblatex}
\renewbibmacro{in:}{%
  \ifentrytype{article}{}{\printtext{\bibstring{in}\intitlepunct}}}
\bibliography{biblio}

\usepackage[colorlinks,citecolor=blue,urlcolor=blue,linkcolor=RawSienna]{hyperref}
\usepackage{hypernat}
\usepackage{datetime}

\DeclareMathAlphabet\EuRoman{U}{eur}{m}{n}
\SetMathAlphabet\EuRoman{bold}{U}{eur}{b}{n}

\usepackage{euscript,microtype}

\usepackage[capitalize]{cleveref}

\crefname{assumption}{Assumption}{Assumptions}
\crefname{claim}{Claim}{Claims}

\makeatletter
\let\reftagform@=\tagform@
\def\tagform@#1{\maketag@@@{\ignorespaces\textcolor{gray}{(#1)}\unskip\@@italiccorr}}
\renewcommand{\eqref}[1]{\textup{\reftagform@{\ref{#1}}}}
\makeatother

\definecolor{WowColor}{rgb}{.75,0,.75}
\definecolor{SubtleColor}{rgb}{0,0,.50}

\newcounter{margincounter}

\declaretheorem[style=plain,numberwithin=section,name=Theorem]{theorem}
\declaretheorem[style=plain,sibling=theorem,name=Lemma]{lemma}

\declaretheorem[style=plain,sibling=theorem,name=Corollary]{corollary}
\declaretheorem[style=plain,sibling=theorem,name=Claim]{claim}

\declaretheorem[style=definition,sibling=theorem,name=Definition]{definition}

\declaretheorem[style=definition,sibling=theorem,name=Example]{example}

\crefformat{conditionINNER}{#2(#1)#3}
\crefmultiformat{conditionINNER}
  {(#2#1#3)}
  { and~(#2#1#3)}
  {, (#2#1#3)}
  { and~(#2#1#3)}
\Crefformat{conditionINNER}{Condition~#2(#1)#3}
\Crefmultiformat{conditionINNER}
  {Conditions~(#2#1#3)}
  { and~(#2#1#3)}
  {, (#2#1#3)}
  { and~(#2#1#3)}

\declaretheoremstyle[
    spaceabove=-6pt,
    spacebelow=6pt,
    headfont=\normalfont\bfseries,
    bodyfont = \normalfont,
    postheadspace=1em,
    qed=$\square$,
    headpunct={{}}]{myproofstyle}

\numberwithin{equation}{section}
\numberwithin{theorem}{section}

\usepackage{amssymb}
\usepackage{mathrsfs}
\usepackage{leftidx}

\usepackage{setspace}

\def\[#1\]{\begin{align}#1\end{align}}
\def\*[#1\]{\begin{align*}#1\end{align*}}

\newcommand{\Reals}{\mathbb{R}}
\newcommand{\Nats}{\mathbb{N}}

\newcommand{\PosReals}{\Reals_{> 0}}

\newcommand{\dee}{\mathrm{d}}

\DeclareMathOperator*{\newlim}{\mathrm{lim}\vphantom{\mathrm{infsup}}}

\DeclareMathOperator*{\newmax}{\mathrm{max}\vphantom{\mathrm{infsup}}}
\DeclareMathOperator*{\newinf}{\mathrm{inf}\vphantom{\mathrm{infsup}}}
\DeclareMathOperator*{\newsup}{\mathrm{sup}\vphantom{\mathrm{infsup}}}
\renewcommand{\lim}{\newlim}

\renewcommand{\max}{\newmax}
\renewcommand{\inf}{\newinf}
\renewcommand{\sup}{\newsup}

\newcommand{\cF}{\mathcal F}

\newcommand{\BorelSets}[1]{\mathcal{B}[#1]}
\newcommand{\NSE}[1]{{^{*}#1}}
\newcommand{\ST}{\mathsf{st}}

\newcommand{\PowerSet}{\mathscr{P}}

\newcommand{\NS}[1]{\mathrm{NS}(#1)}

\newcommand{\cA}{\mathcal{A}}

\newcommand{\cC}{\mathcal{C}}

\newtheorem{open problem}{Open Problem}

\newcommand{\Loeb}[1]{\overline{#1}}

\newcommand{\interior}[1]{%
  {\kern0pt#1}^{\mathrm{o}}%
}

\newcommand{\refproof}[1]{See \cref{#1} for \IfSubStr{#1}{,}{proofs}{a proof}. }

\newif\iflongform
\longformtrue

\iflongform

\else

\fi

\makeatletter
\providecommand*{\toclevel@definition}{0}
\providecommand*{\toclevel@theorem}{0}
\providecommand*{\toclevel@lemma}{0}
\makeatother

\title[Finitely-additive, countably-additive and internal probability measures.]
{
Finitely-additive, countably-additive and internal probability measures. \footnote{\textit{AMS subject code}: 03H05,26E35,28E05,60B10}
}

\newcommand{\cB}{\mathscr{B}}
\newcommand{\Lip}[1]{\mathcal{L}_{1}(#1)}
\newcommand{\FM}[1]{\mathcal{M}(#1)}
\newcommand{\PM}[1]{\mathcal{M}_{1}(#1)}
\newcommand{\topology}{\mathcal{T}}
\newcommand{\pd}[1]{{#1}_p}
\newcommand{\ipd}[1]{{#1}^{p}}
\newcommand{\boundary}{\partial}

\newcommand{\comp}[1]{\hat{#1}}

\begin{document}

\author[H.~Duanmu]{Haosui Duanmu}
\address{University of Toronto, Department of Statistics}
\email{duanmuhaosui@hotmail.com}
%
%
\author[W.~Weiss]{William Weiss}
\address{University of Toronto, Department of Mathematics}
\email{weiss@math.utoronto.ca}
\maketitle

%
\begin{abstract}
We discuss two ways to construct standard probability measures, called push-down measures, from internal probability measures. We show that the Wasserstein distance between an internal probability measure and its push-down measure is infinitesimal.
As an application to standard probability theory, we show that
every finitely-additive Borel probability measure $P$ on a separable metric space is a limit of a sequence of countably-additive Borel probability measures $\{P_n\}_{n\in \Nats}$ in the sense that $\int f \dee P=\lim_{n\to \infty} \int f\dee P_n$ for all bounded uniformly continuous real-valued function $f$ if and only if the space is totally bounded. 

\end{abstract}

\section{Introduction}
One of the foundational problems in probability theory is to study the connection between finitely-additive measures and countably additive measures (see, e.g., \citet{yoshida52}). In contrast to Prokhorov's theorem and the Vitali-Hahn-Saks theorem which state that a sequence of countably additive measures converge to a countably additive measure under regularity conditions, we prove that, for every finitely-additive probability measure $P$ on a totally bounded separable metric space, there is a sequence of countably additive probability measures $\{P_n\}_{n\in \Nats}$ such that $\int f \dee P_n\to \int f\dee P$ for every bounded uniformly continuous real-valued function. On the other hand, unlike the Portmanteau lemma, such convergence fails for merely bounded continuous functions, showing that the hypothesis of the Portmanteau theorem is sharp. 

In \cref{secpre}, we give a gentle introduction to nonstandard analysis as well as nonstandard measure theory. 
Nonstandard measure theory provides powerful machinery to study this problem. 
On one hand internal measures have the same first-order logical properties as finitely-additive measures. On the other hand internal measures can be easily extended to countably additive measures (Loeb measure), using Loeb's construction in \citet{Loeb75}. We can also reverse the procedure to construct finitely-additive probability measures from internal probability measures or Loeb measures. There is a rich literature on constructing standard measures using the standard part map and Loeb measures in very general settings (see, e.g., \citet{anderson87}, \citet{Lindpushdown}, \citet{renderpush}, \citet{aldazdissertation}, \citet{rosscompact}, \citet{aldazrepresent} and \citet{rossinfinite}). On the other hand, by the transfer principle, we can construct finitely-additive probability measures from internal probability measures. 
In this paper, we establish some connections between finitely-additive measures and countably-additive measures by studying the relation between these two forms of push-down, and in doing so, we prove the theorem mentioned in the first paragraph, which has no known standard proof. 

In \cref{starwass}, we define the Wasserstein distance between finitely-additive measures and show that the Wasserstein distance, $\NSE{W}(\nu,\NSE{P})$, between a finitely additive measure $P$ and an internal probability measure $\nu$ is infinitesimal provided that the underlying space is a bounded $\sigma$-compact metric space and $\NSE{P}(\NSE{A})\approx \nu(\NSE{A})$ for all Borel sets $A$.
In contrast to this, we give an example of an internal probability measure $\nu$ with $\nu(\NSE{A})\approx \lambda(A)$ for $A\in \BorelSets{[0,1]}$ where $\lambda$ denote the Lebesgue measure on $[0,1]$. Meanwhile, we have
\[
\sup_{B\in \NSE{\BorelSets {[0,1]}}}|\nu(B)-\NSE{\lambda}(B)|=1.
\]

Given an internal probability measure $\nu$ on $(\NSE{X},\NSE{\BorelSets X})$, the internal push-down measure $\ipd{\nu}$ is a finitely-additive measure on $(X,\BorelSets X)$ defined as $\ipd{\nu}(A)=\ST(\nu(\NSE{A}))$ and the external push-down measure $\pd{\nu}$ is a countably additive measure on $(X,\BorelSets X)$ defined as $\pd{\nu}(A)=\Loeb{\nu}(\ST^{-1}(A))$ where $\Loeb{\nu}$ denote the Loeb extension of $\nu$. \citet{nscredible} showed that $\NSE{W}(\nu, \NSE{\pd{\nu}})\approx 0$ if the underlying space is compact. In \cref{stcharge}, we generalize this result to bounded $\sigma$-compact spaces. We also show that $\NSE{W}(\nu,\ipd{\nu})\approx 0$. Thus, the Wasserstein distance is only a pseudometric on the space of all finitely-additive probability measures. 

There exists a rich literature on studying the relationship between finitely-additive probability measures and countably additive probability measures (e,g., see \citet{yoshida52} and \citet{teddy99}). For an uncountable $\sigma$-algebra $\cF$, the set of finitely-additive probability measures can be viewed as a subset of the compact product space $[0,1]^{\cF}$. Teddy Seidenfeld pointed out that the set of finitely-additive probability measures with finite support forms a dense subset of the set of all finitely-additive probability measures. Thus, every finitely-additive probability measure is an accumulation point of a set of countably additive probability measures. However, no point in $[0,1]^{\cF}$ has a countable local base hence we can not conclude that every finitely-additive probability measure is the limit of a countable sequence of countably additive probability measures under pointwise convergence.

In \cref{weakcharge}, we show that, for every finitely-additive probability measure $P$, there is a sequence of countably-additive probability measures $\{P_n\}_{n\in \Nats}$ such that $\int f \dee P_n$ converges to $\int f\dee P$ for every bounded uniformly continuous real-valued function $f$. We denote such convergence by weak convergence. 
In \cref{mikexample}, we show that our theorem fails if we replace bounded uniformly continuous real-valued function by merely bounded continuous real-valued function, hence our result is sharp. 
We conclude with a nonstandard characterization of weak convergence to finitely-additive probability measures, which is similar in spirit to 
Theorem 4 in \citet{nsweak}.

\section{Preliminaries}\label{secpre}
In this section, we give a short introduction to nonstandard analysis. A large part of this introduction is taken from the preliminary section in \citep{Keisler87}. For a detailed introduction to nonstandard models, we recommend the first four chapters of \citep{NSAA97}. 

Given any set $S$ containing $\Reals$ as a subset, the superstructure $\mathbb{V}(S)$ over $S$ is defined as
\begin{enumerate}
\item $\mathbb{V}_{1}(S)=S$
\item $\mathbb{V}_{n+1}(S)=\mathbb{V}_{n}(S)\cup \PowerSet(S)$
\item $\mathbb{V}(S)=\bigcup_{n\in \Nats}V_{n}(S)$
\end{enumerate}
The starting point of nonstandard analysis is to construct a set $\NSE{\Reals}\supset \Reals$ and a mapping $\ast: \mathbb{V}(\Reals)\mapsto \mathbb{V}(\NSE{\Reals})$ with three basic properties. We first state the following two basic notions from mathematical logic. 
A \emph{formula} is a statement $\phi$ built up from equality and $\in$ relations $x=y$, $x\in y$, the connectives $\wedge, \vee,\neg$ and bounded quantifiers $(\forall x\in y),(\exists x\in y)$. An \emph{internal object} is an element of the set $\bigcup\{\NSE{A}: A\in \mathbb{V}(S)\}$. A set in $\mathbb{V}(S)$ which is not internal is called \emph{external}. We now state the three basic properties.
\begin{enumerate}
\item \textbf{Extension Principle} $\NSE{S}$ is a proper extension of $S$ and $\NSE{s}=s$ for all $s\in S$. 
\item \textbf{Transfer Principle} Let $S_1,\dotsc,S_n\in \mathbb{V}(S)$. Any formula which is true of $S_1,\dotsc,S_n$ is true of $\NSE{S_1},\dotsc,\NSE{S_n}$. 
\item \textbf{$\kappa$-Saturation Principle} Let $\kappa$ be a cardinal number and let $\cF$ be a collection of internal sets. If $\cF$ has the finite intersection property with cardinality no more than $\kappa$, then the total intersection of $\cF$ is non-empty.
\end{enumerate}  
An internal set $A$ is a \emph{hyperfinite set} if there exists an internal bijection $f$ between $A$ and $\{n\in \NSE{\Nats}: n\leq N_0\}$ for some $N_0\in \NSE{\Nats}$.        

In this paper, the nonstandard model is as saturated as it needs to be.
Let $(X,\topology)$ be a topological space. The monad of a point $x\in X$ is the set $\bigcap_{x\in U\in \topology}\NSE{U}$.
An element $x\in \NSE{X}$ is \emph{near-standard} if it is in the monad of some $y\in X$. We say $y$ is the standard part of $x$ and write $y=\ST(x)$. We use $\NS{\NSE{X}}$ to denote the collection of near-standard elements of $\NSE{X}$ and we say $\NS{\NSE{X}}$ is the \emph{near-standard part} of $\NSE{X}$. The standard part map $\ST$ is a function from $\NS{\NSE{X}}$ to $X$. For a metric space $X$, two elements $x,y\in \NSE{X}$ are \emph{infinitely close} if $\NSE{d}(x,y)\approx 0$. For two elements $a,b\in \NSE{\Reals}$, we write $a\lessapprox b$ to mean $a<b$ or $a\approx b$.

Let $X$ be a topological space. The $\sigma$-algebra on $X$ is always taken to be the Borel $\sigma$-algebra and is denoted by $\BorelSets X$. We use  $\PM{X}$ to denote the collection of all countably additive probability measures on $(X,\BorelSets X)$ and let $\FM{X}$ denote the collection of all \emph{charges}, that is, finitely-additive probability measures, on $(X,\BorelSets X)$.
An internal probability measure $\mu$ on $(\NSE{X},\NSE{\BorelSets X})$ is an element of $\NSE{\FM{X}}$. Namely, an internal probability measure $\mu$ on $(\NSE{X},\NSE{\BorelSets X})$ is an internal function from $\NSE{\BorelSets X}\to \NSE{[0,1]}$ such that
\begin{enumerate}
\item $\mu(\emptyset)=0$;
\item $\mu(\NSE{X})=1$; and
\item For $A,B\in \NSE{\BorelSets X}$ with $A\cap B=\emptyset$, $\mu(A\cup B)=\mu(A)+\mu(B)$. 
\end{enumerate}
We use $(\NSE{X},\Loeb{\NSE{\BorelSets X}}, \Loeb{\mu})$ to denote the Loeb extension of the internal probability space $(\NSE{X},\NSE{\BorelSets X},\mu)$.

\section{Wasserstein Metric}\label{starwass}
Let $P$ be a charge on $(X,\BorelSets X)$ and let $\nu$ be an internal probability measure on $(\NSE{X},\NSE{\BorelSets X})$.
Suppose $\NSE{P}(\NSE{A})\approx \nu(\NSE{A})$ for all $A\in \BorelSets X$. We investigate the relation between $\NSE{P}$ and $\nu$.
Note that if $\nu=\NSE{P_1}$ for some charge $P_1$ on $(X,\BorelSets X)$, by the transfer principle, it is easy to see that $P_1=P$.
So we are interested in the case where $\nu$ is not the nonstandard extension of any standard charge.

Integration with respect to charges is similar to integration with respect to countably additive probability measures except that we only have finite additivity. However, we need to be careful about what functions are integrable. We quote the following result regarding integrability of charges.

\begin{lemma}[{\citep[][Corollary.~4.5.9]{charge}}]\label{bdmsrb}
Let $P$ be a charge on $(X,\BorelSets X)$. Let $f$ be a bounded real-valued measurable function on $X$. Then $f$ is integrable with respect to $P$.
\end{lemma}

The Wasserstein distance is usually defined for countably additive probability measures. In this paper, we extend the definition of Wasserstein distance to charges.

\begin{definition}\label{defnwas}
Let $\mu$ and $\nu$ be two charges on some bounded metric space $(Y,d)$ with Borel $\sigma$-algebra $\BorelSets Y$. The Wasserstein distance between $\mu$ and $\nu$ is given by
\[
W(\mu,\nu)=\sup\{|\ \int f\dee \mu-\int f \dee \nu\ |: f\in \Lip{Y}\}
\]
where $\Lip{Y}$ denote the set of $1$-Lipschitz functions from $Y$ to $\Reals$, i.e those functions $f$ such that $|f(x)-f(y)|\leq d(x,y)$ for all $x,y\in Y$.
\end{definition}
As $Y$ is bounded, every $f\in \Lip{Y}$ is bounded measurable. Thus, by \cref{bdmsrb}, the Wasserstein metric on charges is well-defined. The following two lemmas provide a sufficient criterion to establish that the Wasserstein distance between two given internal probability measures is infinitesimal.

\begin{lemma}\label{stareq}
Let $(X,d)$ be a bounded $\sigma$-compact metric space.
Let $P$ be a countably additive probability measure on $(X,\BorelSets X)$ and let $\nu$ be an internal probability measure on $(\NSE{X},\NSE{\BorelSets X})$.
Suppose for every $n\in \Nats$, there is a countable partition $\{V^{n}_i:i\in \Nats\}$ of $X$
consisting of non-empty Borel sets with diameters no greater than $\frac{1}{n}$ such that $\NSE{P}(\NSE{V^{n}_{i}})\approx \nu(\NSE{V^{n}_{i}})$.
Then $\NSE{W}(\nu,\NSE{P})\approx 0$.
\end{lemma}
\begin{proof}
Fix $n\in \Nats$. Let $\{V^{n}_i:i\in \Nats\}$ be a countable partition of $X$ consisting of non-empty Borel sets
with diameters no greater than $\frac{1}{n}$ such that $\NSE{P}(\NSE{V^{n}_{i}})\approx \nu(\NSE{V^{n}_{i}})$.
For every $i\in \Nats$, let $B_{i}$ denote the set of all internal functions $g: \NSE{\Nats}\mapsto \NSE{\BorelSets X}$ such that
\begin{enumerate}
\item $g(i)=\NSE{V^{n}_{i}}$;
\item $(\forall k\in \NSE{\Nats})(\sup\{\NSE{d}(x,y)|x,y\in g(k)\}\leq \frac{1}{n})$;
\item $(\forall k\in \NSE{\Nats})(|\nu(g(k))-\NSE{P}(g(k))|\leq \frac{1}{i})$.
\end{enumerate}
Let $\cB$ be the collection of all $B_{i}$. Then $\cB$ has countable cardinality and the finite intersection property.
By the saturation principle, there is an internal function $g_0$ which is an element of $B_i$ for all $i\in \Nats$.
Note that $g_0(i)=\NSE{V^{n}_i}$ for all $i\in \Nats$ and $\nu(g_0(k))\approx \NSE{P}(g_0(k))$ for all $k\in \NSE{\Nats}$.
By overspill, there is a $K\in \NSE{\Nats}\setminus \Nats$ such that $g_{0}(i)\neq \emptyset$ for all $i\leq K+1$.
For any $i\leq K$, pick $x_i\in g_{0}(i)$.
Pick some $F\in \NSE{\Lip{X}}$.
As $F\in \NSE{\Lip{X}}$ and the diameter of every $g_{0}(i)$ is no greater than $\frac{1}{n}$,
we have
\[
(\forall i\leq K)(\forall x\in g_{0}(i))(|F(x_i)-F(x)|\leq\frac{1}{n}).
\]
Finally, we pick some $x_{K+1}\in \NSE{X}\setminus \bigcup_{i\leq K}g_{0}(i)$.
As $X$ is bounded and $F\in \NSE{\Lip{X}}$, we know that the function $|F(x)-F(x_{K+1})|$ is bounded by a standard real number.
Define $g_1: \{1,2,\dotsc,K+1\}\to \NSE{\BorelSets X}$ to be the internal function such that $g_{1}(i)=g_{0}(i)$ for all $i\leq K$ and $g_{1}(K+1)=\NSE{X}\setminus \bigcup_{i\leq K}g_{0}(i)$.
Hence we have
\[
&|\ \int_{\NSE{X}} F(x)\ \nu(\dee x)-\sum_{i\leq K+1}\int_{g_{1}(i)} F(x_i)\ \nu(\dee x)\ |\\
&\leq \sum_{i\leq K}\int_{g_{1}(i)}|F(x)-F(x_i)|\ \nu(\dee x)+\int_{g_{1}(K+1)}|F(x)-F(x_{K+1})|\ \nu(\dee x)\\
&\lessapprox \frac{1}{n}.
\]
Similarly, we have
\[
|\ \int_{\NSE{X}} F(x)\ \NSE{P}(\dee x)-\sum_{i\leq K+1}\int_{g_{1}(i)} F(x_i)\ \NSE{P}(\dee x)\ |\lessapprox \frac{1}{n}.
\]
We now compare $\sum_{i\leq K+1}\int_{g_{1}(i)} F(x_i)\ \nu(\dee x)$ and $\sum_{i\leq K+1}\int_{g_{1}(i)} F(x_i)\ \NSE{P}(\dee x)$.
Note that
\[
&|\ \sum_{i\leq K+1}\int_{g_{1}(i)} F(x_i)\ \nu(\dee x)-\sum_{i\leq K+1}\int_{g_{1}(i)} F(x_i)\ \NSE{P}(\dee x)\ |\\
&=|\ \sum_{i\leq K+1}F(x_i)(\nu(g_{1}(i))-\NSE{P}(g_{1}(i)))\ |\\
&=|\ \sum_{i\leq K+1}(F(x_1)+k_i)(\nu(g_{1}(i))-\NSE{P}(g_{1}(i)))\ |\\
&=|\ \sum_{i\leq K+1}F(x_1)(\nu(g_{1}(i))-\NSE{P}(g_{1}(i)))+\sum_{i\leq K+1}k_i(\nu(g_{1}(i))-\NSE{P}(g_{1}(i)))\ |\\
&=|\ \sum_{i\leq K+1}k_i(\nu(g_{1}(i))-\NSE{P}(g_{1}(i)))\ |.
\]
where $k_i$ is the difference between $F(x_i)$ and $F(x_1)$.

As $X$ is bounded and $F\in \NSE{\Lip{X}}$, we know that $k_i\in \NS{\NSE{\Reals}}$ for all $i\leq K+1$.
Suppose $|\ \sum_{i\leq K+1}k_i(\nu(g_{1}(i))-\NSE{P}(g_{1}(i)))\ |\approx 0$ then we have
\[
|\ \int_{\NSE{X}} F(x)\ \nu(\dee x)-\int_{\NSE{X}} F(x)\ \NSE{P}(\dee x)\ |\lessapprox \frac{2}{n}.
\]
As $n$ is arbitrary, we know that $\int_{\NSE{X}} F(x)\ \nu(\dee x) \approx \int_{\NSE{X}} F(x)\ \NSE{P}(\dee x)$ hence we have $\NSE{W}(\nu,\NSE{P})\approx 0$ by \cref{defnwas}.
Thus, in order to finish the proof, it is sufficient to prove $|\ \sum_{i\leq K+1}k_i(\nu(g_{1}(i))-\NSE{P}(g_{1}(i)))\ |\approx 0$.
\begin{claim} $\sum_{i\leq K+1}k_i(\nu(g_{1}(i))-\NSE{P}(g_{1}(i)))\approx 0$. \end{claim}
\begin{proof}
Pick some $k\in \Nats$. As $P$ is countably additive, there exists $m\in \Nats$ such that $\sum_{i\leq m}\nu(g_{1}(i))\geq 1-\frac{1}{k}$ and $\sum_{i\leq m}\NSE{P}(g_{1}(i))\geq 1-\frac{1}{k}$.
Thus we have
\[
&\sum_{i\leq K+1}k_i(\nu(g_{1}(i))-\NSE{P}(g_{1}(i)))\\
&=\sum_{i\leq m}k_i(\nu(g_{1}(i))-\NSE{P}(g_{1}(i)))+\sum_{m<i\leq K+1}k_i(\nu(g_{1}(i))-\NSE{P}(g_{1}(i)))\\
&\approx \sum_{m<i\leq K+1}k_i(\nu(g_{1}(i))-\NSE{P}(g_{1}(i)))\leq \frac{\max\{k_i:i\leq K+1\}}{k}.
\]
As $\max\{k_i:i\leq K+1\}$ is near-standard and $k$ is arbitrary, we have the desired result.
\end{proof}
Hence we have completed the proof. 
\end{proof}

By using a similar argument as in \cref{stareq}, we obtain the following result for charges.

\begin{lemma}\label{finstareq}
Let $(X,d)$ be a bounded $\sigma$-compact metric space.
Let $P$ be a charge on $(X,\BorelSets X)$ and let $\nu$ be an internal probability measure on $(\NSE{X},\NSE{\BorelSets X})$.
Suppose for every $n\in \Nats$, there is a finite partition $\{V^{n}_i:i\leq N\}$ of $X$
consisting of non-empty Borel sets with diameters no greater than $\frac{1}{n}$ such that $\NSE{P}(\NSE{V^{n}_{i}})\approx \nu(\NSE{V^{n}_{i}})$.
Then $\NSE{W}(\nu,\NSE{P})\approx 0$.
\end{lemma}

The following example shows that \cref{stareq} and \cref{finstareq} are sharp. 

\begin{example}
Let $X=(0,1)$ endowed with the standard metric and Borel $\sigma$-algebra $\BorelSets X$. Let $\nu$ be an internal probability measure concentrates on some infinitesimal $\epsilon$. Let $P$ be a charge with $P((0,1-\frac{1}{n}])=0$ for all $n\in \Nats$ and $P((0,1))=1$. 
Pick $n\in \Nats$. We can pick $m\geq n$ such that $[\frac{1}{m},1-\frac{1}{m}]$ is a non-empty subset of $(0,1)$. We can partition $[\frac{1}{m},1-\frac{1}{m}]$ into $k$ Borel sets with diameter no greater than $\frac{1}{n}$ for some $k\in \Nats$. We denote these sets by $V^{n}_{i}$ for $i\leq k$. Let $V^{n}_{k+j}=[\frac{1}{2^{j}m},\frac{1}{2^{j-1}m})\cup (1-\frac{1}{2^{j-1}m},1-\frac{1}{2^{j}m}]$. Thus, $\{V^{n}_{i}: i\in \Nats\}$ forms a countable partition of $(0,1)$ consisting of Borel sets with diameter no greater than $\frac{1}{n}$. Note that $\nu(\NSE{V^{n}_{i}})=\NSE{P}(\NSE{V^{n}_{i}})=0$.

On the other hand, let $f$ be the identity function from $(0,1)\to (0,1)$. Then we have $\int \NSE{f}(x) \nu(\dee x)\approx 0$ while $\int \NSE{f}(x)\NSE{P}(\dee x)\approx 1$. Hence, we have $\NSE{W}(\nu,\NSE{P})\gtrapprox 1$. 
\end{example}

We conclude this section with the following theorem, which is a direct consequence of \cref{finstareq}. 

\begin{theorem}\label{starall}
Let $(X,d)$ be a bounded $\sigma$-compact metric space.
Let $P$ be a charge on $(X,\BorelSets X)$ and let $\nu$ be an internal probability measure on $(\NSE{X},\NSE{\BorelSets X})$.
Suppose $\nu(\NSE{B})\approx \NSE{P}(\NSE{B})$ for all $B\in \BorelSets X$.
Then $\NSE{W}(\nu,\NSE{P})\approx 0$.
\end{theorem}

\section{Construction of Standard Charges}\label{stcharge}

One of the strength of nonstandard analysis is its ability to construct exotic standard objects.
In this section, we discuss two procedures on constructing standard measurs/charges on metric spaces.
We also establish some connections between standard objects obtained from these two different approaches.

We begin with an internal probability measure and then use the transfer principle to push it down to get a charge.

\begin{definition} \label{intpushdown}
Let $(X,\BorelSets X)$ be a measurable space and let $\nu$ be an internal probability measure on $(\NSE{X},\NSE{\BorelSets X})$. Its \emph{internal push-down} is a function $\ipd{\nu}: \BorelSets X\mapsto [0,1]$ defined by $\ipd{\nu}(A)=\ST({\nu(\NSE{A})})$.
\end{definition}

The following lemma follows immediately from \cref{intpushdown}.

\begin{lemma} \label{pushdownlemma}
Let $\nu$ be an internal probability measure on $(\NSE{X},\NSE{\BorelSets X})$. Then its internal push-down measure $\ipd{\nu}$ is a charge on $(X,\BorelSets X)$.
\end{lemma}

In most cases, $\ipd{\nu}$ is not a countably additive probability measure. Moreover, the nonstandard extension of $\ipd{\nu}$ is usually not the same as $\nu$.


The following theorem is a direct consequence of \cref{starall}. 

\begin{theorem}\label{intpdclose}
Let $(X,d)$ be a bounded $\sigma$-compact metric space.
Let $\nu$ be an internal probability measure on $(\NSE{X},\NSE{\BorelSets X})$.
Let $\ipd{\nu}$ be the internal push-down of $\nu$.
Then $\NSE{W}(\nu,\NSE{\ipd{\nu}})\approx 0$.
\end{theorem}

Although the internal push-down of an internal probability measure always exists, it is merely finitely-additive in most cases. The properties of internal push-down are closely related to the internal probability measure via transfer principle. To get a countably additive probability measure, we shall use the standard part map to push down the Loeb measure.

\begin{definition}\label{defnpushdown}
Let $X$ be a Hausdorff space with Borel $\sigma$-algebra $\BorelSets X$,
let $\nu$ be an internal probability measure defined on $(\NSE{X},\NSE{\BorelSets X})$,
and let
\[
\cC=\{C\subset X: \ST^{-1}(C) \in \overline{\NSE{\BorelSets X}}\}.
\]
The \emph{external push-down measure} $\pd{\nu}$ is defined on the set $\cC$ by $\pd{\nu}(C)=\overline{\nu}(\ST^{-1}(C))$.
\end{definition}

The following two theorems guarantee that $\ST^{-1}(C)\in \overline{\NSE{\BorelSets X}}$ for all $C\in \BorelSets X$ under moderate assumptions.

\begin{theorem}[{\citep[][Thm.~4.3.2]{NSAA97}}]
Let $X$ be a regular topological space and let $P$ be an internal probability measure on $(\NSE{X},\NSE{\BorelSets X})$. Suppose $\NS{\NSE{X}}\in \overline{\NSE{\BorelSets X}}$. Then $\ST^{-1}(A)\in \overline{\NSE{\BorelSets X}}$ for all $A\in \BorelSets X$ (i.e., $\ST$ is Borel measurable).
\end{theorem}

\begin{theorem}[{\citep[][Thm.~5.6]{Markovpaper}}]
Let $X$ be a Cech-complete Tychnoff space with Borel $\sigma$-algebra $\BorelSets X$. Then $\NS{\NSE{X}}\in \overline{\NSE{\BorelSets X}}$.
\end{theorem}
In particular, we have $\NS{\NSE{X}}\in \overline{\NSE{\BorelSets X}}$ for regular locally compact spaces; for complete metric spaces; and for regular $\sigma$-compact spaces.

For general Hausdorff Borel measurable space $(Y,\BorelSets Y)$, the external push-down measure $\pd{\nu}$ may not be a countably additive probability measure. In fact, if $\Loeb{\nu}(\NS{\NSE{X}})=0$ then $\pd{\nu}$ is a null measure on $(X,\BorelSets X)$. However, when $Y$ is compact, the following theorem guarantees that $\pd{\nu}$ is a countably additive probability measure on $(Y,\BorelSets Y)$.

\begin{theorem}[{\citep[][Thm.~13.4.1]{NDV}}]\label{pushdown}
Let $X$ be a Hausdorff space equipped with Borel $\sigma$-algebra $\BorelSets {X}$,
and let $\nu$ be an internal probability measure defined on $(\NSE{X}, \NSE{\BorelSets{X}})$ with $\overline{\nu}(\NS{\NSE{X}})=1$.
Then the external push-down measure $\pd{\nu}$ of $\nu$ is the completion of a countably additive regular Borel probability measure.
\end{theorem}

Given an internal probability measure $\nu$ on $(\NSE{X},\NSE{\BorelSets X})$. It is easy to see that the total variation distance between $\nu$ and $\NSE{\pd{\nu}}$ may be large. For example, if $\nu$ is an internal probability measure concentrating on some infinitesimal $\epsilon$ then $\pd{\nu}$ is a degenerate probability measure at point $0$. The total variation distance between $\nu$ and $\NSE{\pd{\nu}}$ is $1$ in this case. However, we show that $\nu$ and $\NSE{\pd{\nu}}$ are close in Wasserstein metric. We start by stating the following well-known definition.

\begin{definition}
Let $X$ be a topological space and let $(X,\BorelSets X,P)$ be a Borel probability space. A set $A\subset X$ is a $P$-continuity set if the boundary $\boundary{A}$ is contained in a measure $0$ set.
\end{definition}

Recall that a countably additive probability measure $P$ on a Borel measurable space $(X, \BorelSets X)$ is Radon provided that $P(E)=\sup\{P(K): K\ \text{compact and }\ K\subset E \}$. The following result, due to Robert Anderson, is the first major result on representing standard measures using nonstandard measures via the standard part map.

\begin{lemma} [{\citep[][Thm.~4.1]{NSAA97}}]\label{stpreserve} Let $(X,\BorelSets X,P)$ be a countably additive Radon probability measure. Then $\ST$ is measure-preserving from $(\NSE{X},\overline{\NSE{\BorelSets X}},\overline{\NSE{P}})$ to $(X,\BorelSets X,P)$, i.e $P(A)=\overline{\NSE{P}}(\ST^{-1}(A))$ for all $A\in \BorelSets X$.
\end{lemma}

The following result gives a nonstandard characterization of compact sets.

\begin{theorem}[{\citep[][Thm.~3.5.1]{NSAA97}}]\label{nonstcompact}
A set $A\subset X$ is compact if and only if for each $y\in \NSE{A}$, there is an $x\in A$ such that $y$ is in the monad of $x$.
\end{theorem}

\begin{lemma}[{\citep[][Exercise.~4.27]{NSAA97}}]\label{stcompact}
If $X$ is a Hausdorff regular space and $A$ is an internal subset of $\NS{\NSE{X}}$,  then $E=\ST(A)$ is compact.
\end{lemma}

The following result is a partial converse of \cref{stpreserve}. 
It follows immediately from the fact that every Borel probability measure on a Polish space is Radon. 

\begin{lemma}\label{pushradon}
Let $X$ be a $\sigma$-compact metric space with Borel $\sigma$-algebra $\BorelSets X$.
Let $\nu$ be an internal probability measure on $(X,\BorelSets X)$.
Then the push-down measure $\pd{\nu}$ is a Radon measure on $(X,\BorelSets X)$.
\end{lemma}

\begin{lemma}\label{cpartition}
Let $X$ be a $\sigma$-compact metric space with Borel $\sigma$-algebra $\BorelSets X$. Let $\nu$ be an internal probability measure on $(\NSE{X},\NSE{\BorelSets X})$. Suppose $\pd{\nu}$ is a countably additive probability measure on $(X,\BorelSets X)$. For every $n\in \Nats$, there exists a countable partition $\{A_i:i\in \Nats\}$ of $X$ consisting of Borel sets with diameter no greater than $\frac{1}{n}$ such that $\nu(\NSE{A_i})\approx \NSE{\pd{\nu}}(\NSE{A_i})$ for all $i\in \Nats$.
\end{lemma}
\begin{proof}
Pick $n\in \Nats$.
For every $x\in X$, there are uncountably many open balls containing $x$ with diameter no greater than $\frac{1}{n}$.
The boundaries of these open balls form a uncountable collection of disjoint sets.
Thus, for every $x\in X$, we can pick an open ball $U_x$ containing $x$ such that its diameter is no greater than $\frac{1}{n}$ and it is a $\pd{\nu}$-continuity set.
As  $X$ is a Polish space, there is a countable subcollection of $\{U_{x}:x\in X\}$ that covers $X$.
Denote this countable subcollection by $\mathcal{K}_n=\{U_{x_i}:i\in \Nats\}$.

Pick $i,j\leq m$. Note that $\boundary(U_{x_i}\cap U_{x_j})\subset \boundary U_{x_i}\cup \boundary U_{x_j}$ and $\boundary(U_{x_i}\cup U_{x_j})\subset \boundary U_{x_i}\cup \boundary U_{x_j}$.
Hence any finite intersection(union) of elements from $\mathcal{K}_n$ is a bounded open $\pd{\nu}$-continuity set.
For any $i\in \Nats$, let
\[
V_i=U_{x_i}\setminus \bigcup_{j<i}(U_{x_j}\cap U_{x_i}).
\]
The following claim shows that $\{V_i: i\in \Nats\}$ is the desired partition.

\begin{claim}\label{pdcontinuity}
$\nu(\NSE{U})\approx \NSE{\pd{\nu}}(\NSE{U})$ for any open $\pd{\nu}$-continuity set $U$.
\end{claim}
\begin{proof}
Pick any open $\pd{\nu}$-continuity set $U$.
By \cref{pushradon}, $\pd{\nu}$ is a countably additive Radon probability measure.
As $U$ is a $\pd{\nu}$-continuity set, by \cref{stpreserve},we have
\[
\overline{\NSE{\pd{\nu}}}(\ST^{-1}(U))=\pd{\nu}(U)=\pd{\nu}(\overline{U})=\overline{\NSE{\pd{\nu}}}(\ST^{-1}(\overline{U})).
\]
By the construction of $\pd{\nu}$, we also have $\overline{\nu}(\ST^{-1}(U))=\pd{\nu}(U)$ and $\overline{\nu}(\ST^{-1}(\overline{U}))=\pd{\nu}(\overline{U})$.
Thus, we have $\overline{\nu}(\ST^{-1}(U))=\overline{\NSE{\pd{\nu}}}(\ST^{-1}(U))$ and $\overline{\nu}(\ST^{-1}(\overline{U}))=\overline{\NSE{\pd{\nu}}}(\ST^{-1}(\overline{U}))$.
As $U$ is an open set, we have $\ST^{-1}(U)\subset\NSE{U}\cap \NS{\NSE{X}}\subset \ST^{-1}(\overline{U})$.
As $\pd{\nu}$ is a countably additive probability measure, we have $\overline{\nu}(\NS{\NSE{X}})=1$.
Hence $\overline{\nu}(\NSE{U})=\overline{\nu}(\NSE{U}\cap \NS{\NSE{X}})$.
Hence
\[
\nu(\NSE{U})\approx \overline{\nu}(\NSE{U}\cap \NS{\NSE{X}})=\NSE{\pd{\nu}}(\NSE{U})
\]
for all open $\pd{\nu}$-continuity set $U$.
\end{proof}
Hence we have completed the proof. 
\end{proof}

We obtain the following result from \cref{stareq} and \cref{cpartition}.

\begin{theorem}\label{pdclose}
Suppose $X$ is a bounded $\sigma$-compact metric space. Let $\nu$ be an internal probability measure on $(\NSE{X},\NSE{\BorelSets X})$. Suppose $\pd{\nu}$ is a countably additive probability measure on $(X,\BorelSets X)$. Then $\NSE{W}(\nu, \NSE{\pd{\nu}})\approx 0$.
\end{theorem}
The compact version of \cref{cpartition,pdclose} were proved by \citet{nscredible}. The structure of the proofs are similar.

By \cref{intpdclose,pdclose}, we immediately obtain the following result.

\begin{theorem}\label{wassnotdistance}
Suppose $X$ is a bounded $\sigma$-compact metric space.
Let $\nu$ be an internal probability measure on $(\NSE{X},\NSE{\BorelSets X})$.
Let $\pd{\nu}$ and $\ipd{\nu}$ denote the external push-down and internal push-down of $\nu$, respectively.
Suppose $\pd{\nu}$ is a countably additive probability measure on $(X,\BorelSets X)$.
Then $W(\pd{\nu},\ipd{\nu})=0$.
\end{theorem}

As $\pd{\nu}$ and $\ipd{\nu}$ are different objects,  the Wasserstein distance is only a pseudometric on $\FM{X}$. In summary, we have the following result.

\begin{theorem}
Suppose $X$ is a bounded $\sigma$-compact metric space.
Let $\nu$ and $\mu$ be two internal probability measures on $(\NSE{X},\NSE{\BorelSets X})$.
Suppose both $\pd{\nu}$ and $\pd{\mu}$ are countably additive probability measures on $(X,\BorelSets X)$.
Then the following statements are equivalent:

\begin{enumerate}
\item $\pd{\nu}=\pd{\mu}$.
\item $\NSE{W}(\nu,\mu)\approx 0$.
\item $W(\ipd{\nu},\ipd{\mu})=0$
\item $W(\pd{\nu},\ipd{\mu})=0$
\end{enumerate}
\end{theorem}
\begin{proof}
As $\pd{\nu}$ and $\pd{\mu}$ are countably additive probability measures on $(X,\BorelSets X)$, we have $\pd{\nu}=\pd{\mu}$ if and only if $W(\pd{\nu},\pd{\mu})=0$. By \cref{pdclose}, we know that $W(\pd{\nu},\pd{\mu})=0$ if and only if $\NSE{W}(\nu,\mu)\approx 0$. By \cref{intpdclose}, we have $\NSE{W}(\nu,\mu)\approx 0$ if and only if $W(\ipd{\nu},\ipd{\mu})=0$. By \cref{pdclose,wassnotdistance}, we have $\NSE{W}(\nu,\mu)\approx 0$ if and only if $W(\pd{\nu},\ipd{\mu})=0$ hence finishing the proof.
\end{proof}

\section{Weak Convergence of Charges}\label{weakcharge}
In this section, we define a notion of weak convergence for charges and show that every charge is a weak limit of a countable sequence of countably additive probability measures under moderate conditions.

We start by proving the result when the underlying space is a compact metric space.

\begin{lemma}\label{compactcharge}
Let $X$ be a compact metric space equipped with Borel $\sigma$-algebra $\BorelSets X$. For every charge $P$ on $(X,\BorelSets X)$, there is a countably-additive probability measure $\mu$ on $(X,\BorelSets X)$ such that $W(P,\mu)=0$
\end{lemma}
\begin{proof}
Fix a charge $P$ on $(X,\BorelSets X)$. Let $\mu=\pd{(\NSE{P})}$. As $X$ is compact, $\mu$ defines a countably additive probability measure on $(X,\BorelSets X)$. We always have $P=\ipd{(\NSE{P})}$, so by \cref{wassnotdistance}, we have the desired result.
\end{proof}

To generalize \cref{compactcharge} to non-compact spaces, we need several results on nonstandard integration theory.
We start by quoting the following lemma.  

\begin{lemma}[{\citep[][Lemma.~6.5]{nsbayes}}]\label{newpushdownint}
Let $X$ be a compact Hausdorff space equipped with Borel $\sigma$-algebra $\BorelSets {X}$,
let $\nu$ be an internal probability measure on $(\NSE{X},\NSE{\BorelSets {X}})$,and
let $f : X \to \Reals$ be a bounded Borel measurable function. Define $g : \NSE{X} \to \Reals$ by $g(s) = f(\ST(s))$.
Then we have $\int f \dee \pd{\nu}=\int g \,\dee \Loeb{\nu}$.
\end{lemma}


\begin{theorem}\label{pdintcts}
Let $X$ be a compact Hausdorff space equipped with Borel $\sigma$-algebra $\BorelSets {X}$,
let $\nu$ be an internal probability measure on $(\NSE{X},\NSE{\BorelSets {X}})$,and
let $f : X \to \Reals$ be a bounded continuous function. Then we have $\int f \dee \pd{\nu}=\int \NSE{f} \,\dee \Loeb{\nu}$.
\end{theorem}
\begin{proof}
Since $X$ is compact, we can define $g: \NSE{X} \to \Reals$ by $g(s)=f(\ST(s))$. By \cref{newpushdownint}, we have $\int f \dee \nu_{p}=\int g \,\dee \Loeb{\nu}$. As $f$ is continuous, we have $\NSE{f}(x)\approx g(x)$ for all $x\in \NSE{X}$. Thus we have $\int g \,\dee \Loeb{\nu}=\int \NSE{f} \,\dee \Loeb{\nu}$, completing the proof.
\end{proof}

We now consider the relation between internal integration and integration with respect to internal push-down measures.

\begin{theorem}\label{pushdowneq}
Let $X$ be a metric space equipped with Borel $\sigma$-algebra $\BorelSets X$.
Let $\nu$ be an internal probability measure on $(\NSE{X},\NSE{\BorelSets X})$ and let $f: X\mapsto \Reals$ be a bounded Borel measurable function. Then we have
\[
\int_{\NSE{X}} \NSE{f}(x) \nu(\dee x)\approx \int_{X} f(x) \ipd{\nu}(\dee x).
\]
\end{theorem}

\begin{proof}
Fix $\epsilon\in \PosReals$. Let $\{K_1,K_2,\dotsc,K_n\}$ be a partition of a large enough interval of $\Reals$ containing the range of $f$ such that every $K_i\in \{K_1,K_2,\dotsc,K_n\}$ is an interval with diameters no greater than $\epsilon$. For $i\leq n$, let $F_i=f^{-1}(K_i)$. Then $\{F_1,\dotsc,F_n\}\subset \BorelSets X$ is a partition of $X$ such that $|f(x)-f(x')|<\epsilon$ for every $x,x'\in F_i$ for every $i=1,2,\dotsc,n$.
Pick $x_{i}\in F_i$ for every $i=1,2,\dotsc,n$.
Define $g: X\mapsto \Reals$ by letting $g(x)=f(x_i)$ if $x\in F_i$ for every $i=1,2,\dotsc,n$.
Then $g$ is a simple bounded measurable real valued function on $X$.
Thus, by \cref{bdmsrb}, both $g$ and $f$ are integrable with respect to $\ipd{\nu}$.

We now have
$|\ \int_{\NSE{X}} \NSE{f}(x) \nu(\dee x)-\int_{X} f(x) \ipd{\nu}(\dee x)\ |\leq |\int \NSE{f}(x) \nu(\dee x)- \int \NSE{g}(x) \nu(\dee x)|
+|\int \NSE{g}(x) \nu(\dee x)-\int g(x) \ipd{\nu}(\dee x)|+|\int g(x) \ipd{\nu}(\dee x)-\int f(x) \ipd{\nu}(\dee x)|$
where all internal integrals are over $\NSE{X}$ and all standard integrals are over $X$.

By the transfer principle, we have $|\NSE{f}(x)-\NSE{f}(x')|<\epsilon$ for every $x,x'\in \NSE{F_i}$ for every $i=1,2,\dotsc,n$. Thus, we have $|\NSE{f}(x)-\NSE{g}(x)|<\epsilon$ for all $x\in \NSE{X}$.
Thus, we have $|\int_{\NSE{X}} \NSE{f}(x) \nu(\dee x)- \int_{\NSE{X}} \NSE{g}(x) \nu(\dee x)|\leq \int_{\NSE{X}}|\NSE{f}(x)-\NSE{g}(x)|\nu(\dee x)<\epsilon$. Similarly, we have $|\int_{X} g(x) \ipd{\nu}(\dee x)-\int_{X} f(x) \ipd{\nu}(\dee x)|<\epsilon$.  For the term $|\int_{\NSE{X}} \NSE{g}(x) \nu(\dee x)-\int_{X} g(x) \ipd{\nu}(\dee x)|$, we have:
\[
\int_{\NSE{X}} \NSE{g}(x) \nu(\dee x)&=\sum_{i=1}^{n}\int_{\NSE{F_i}} \NSE{g}(x) \nu(\dee x)\\
&=\sum_{i=1}^{n}\NSE{f}(x_i)\nu(\NSE{F_i})
=\sum_{i=1}^{n}f(x_i)\nu(\NSE{F_i})
\approx \sum_{i=1}^{n}f(x_i)\ipd{\nu}(F_i)\\
&=\sum_{i=1}^{n}\int_{F_i} g(x) \ipd{\nu}(\dee x)=\int_{X} g(x) \ipd{\nu}(\dee x)
\]
Thus, we have $|\int_{\NSE{X}} \NSE{f}(x) \nu(\dee x)-\int_{X} f(x) \ipd{\nu}(\dee x)|\lessapprox 2\epsilon$. As $\epsilon$ is arbitrary, we have $\int_{\NSE{X}} \NSE{f}(x) \nu(\dee x)\approx \int_{X} f(x) \ipd{\nu}(\dee x)$.
\end{proof}

The following corollary is a direct consequence of \cref{pdintcts,pushdowneq}:
\begin{corollary}\label{pdinteq}
Let $X$ be a compact Hausdorff space equipped with Borel $\sigma$-algebra $\BorelSets {X}$,
let $\nu$ be an internal probability measure on $(\NSE{X},\NSE{\BorelSets {X}})$,and
let $f : X \to \Reals$ be a bounded continuous function. Then we have $\int f \dee \pd{\nu}=\int f \dee \ipd{\nu}$.
\end{corollary}

Before we establish the main result of this section, we introduce the following definition.

\begin{definition}\label{defwkconverge}
A sequence of charges $P_{n}$ is said to converge weakly to a charge $P$ if $\int f \dee P_{n}\to \int f \dee P$ for all bounded uniformly continuous real-valued function $f$.
\end{definition}

If $\{P_n\}_{n\in \Nats}$ and $P$ are countably additive probability measures, we have the following well-known result.
\begin{theorem}[The Portmanteau Theorem {\citep[][Thm.~10.1.1]{jeff06}}]\label{thmport}
Suppose $\{P_n\}_{n\in \Nats}$ is a sequence of countably additive probability measures and $P$ is a countably additive measure. Then the following are equivalent:
\begin{enumerate}
\item $\int f \dee P_n\to \int f \dee P$ for all bounded continuous functions $f$.
\item $\int f \dee P_n\to \int f \dee P$ for all bounded uniformly continuous functions $f$.
\item $P_n(A)\to P(A)$ for all $P$-continuity sets $A$.
\end{enumerate}
\end{theorem}

We also need the following theorem from point-set topology. 

\begin{theorem}[{\citep[][Thm.~1]{extfunction}}]\label{lipextend}
Let $f$ be a real-valued function defined on some closed subset $E$ of a metric space $S$.
Suppose $f$ is $M$-Lipschitz continuous for some $M\in \Reals$.
Then $f$ can be extended to a $M$-Lipschitz continuous function on $S$.
\end{theorem} 

We are now able to establish the main result of this section.

\begin{theorem}\label{wkconverge}
Let $X$ be a separable metric space equipped with Borel $\sigma$-algebra $\BorelSets X$. Let $P$ be a charge on $(X,\BorelSets X)$. There is a sequence $\{P_n\}_{n\in \Nats}$ of finitely supported probability measures that converges to $P$ weakly if and only if $X$ is totally bounded. 
\end{theorem}
\begin{proof}
Suppose $X$ is totally bounded. 
In the case that $X$ is compact, the result follows immediately from \cref{pdinteq} by letting $\nu=\NSE{P}$.

In the case that $X$ is not compact, let $\comp{X}$ denote the completion of $X$. Then $\comp{X}$ is a compact space. 
We can extend $P$ to a charge on $(\comp{X},\BorelSets {\comp{X}})$ by letting $P(A)=P(A\cap \comp{X})$ for $A\in \BorelSets {\comp{X}}$. 
As $\comp{X}$ is compact, by \cref{pushdown}, 
$\pd{(\NSE{P})}$ is a countably additive probability measure on $(\comp{X},\BorelSets {\comp{X}})$.
Pick $n\in \Nats$.  As $X$ is totally bounded and dense in $\comp{X}$, there is a finite collection of open balls such that their closure covers $\comp{X}$.
Thus, we can decompose $\comp{X}$ into finitely many mutually disjoint Borel sets $\{B_{i}^{n}: i\leq m\}$ with diameters no greater than $\frac{1}{n}$ and each $B_i$ contains at least one element $x_i$ from $X$.
Define a finitely supported probability measure $P_n$ on $(\comp{X},\BorelSets {\comp{X}})$ by letting $P_n(\{x_i\})=\pd{\NSE{P}}(B_i)$ for every $i\leq m$.  
\begin{claim}\label{wccompact}
The sequence $\{P_n\}_{n\in \Nats}$ converges to $\pd{\NSE{P}}$ weakly. 
\end{claim}
\begin{proof}
Pick a uniformly continuous function $f: \comp{X}\to \Reals$ and a positive $\epsilon\in \Reals$. 
There exists $j\in \Nats$ such that $|f(x)-f(y)|<\epsilon$ if $d(x,y)<\frac{1}{j}$ for all $x,y\in \comp{X}$. 
Then, for every $n\geq j$, we have
\[
&|\int_{\comp{X}}f(x) \dee \pd{\NSE{P}}-\int_{\comp{X}}f(x) \dee P_{n}|\\ 
&=|\sum \int_{B_{i}^{n}}f(x) \dee \pd{\NSE{P}}- \sum \int_{B_{i}^{n}}f(x) \dee P_{k}|\\
&=|\sum \int_{B_{i}^{n}}f(x) \dee \pd{\NSE{P}}- \sum \int_{B_{i}^{n}}f(x_i) \dee \pd{\NSE{P}}|\\
&\leq \int_{B_{i}^{n}}|f(x)-f(x_i)|\dee \pd{\NSE{P}}<\epsilon
\]
By \cref{thmport}, we have the desired result. 
\end{proof}

We now show that the sequence $(P_n)_{n\in \Nats}$ converges to $P$ weakly.
Pick a bounded uniformly continuous function $f: X\to \Reals$. We can extend $f$ to a bounded continuous function $\comp{f}: \comp{X}\to \Reals$. By \cref{pdinteq,wccompact}, we have $\int_{\comp{X}} \hat{f} \dee P_{n}\to \int_{\comp{X}} \hat{f} \dee \pd{(\NSE{P})}=\int_{\comp{X}} \hat{f} \dee P$. As the supports of $P_n$ and $P$ are subsets of $X$, we have $\int_{\comp{X}}\hat{f} \dee P_{n}=\int_{X} f \dee P_{n}$ for all $n\in \Nats$ and $\int_{\comp{X}}\hat{f} \dee P=\int_{X} f \dee P$. Thus the sequence $\{P_n\}_{n\in \Nats}$ converges to $P$ weakly.

\vspace{0.5cm}

Now Suppose $X$ is not totally bounded. 
\begin{claim}\label{countdiscrete}
There exists a closed countably infinite discrete subset $Y$ of $X$ such that, for every pair of distinct points $y_1,y_2\in Y$, we have $d(y_1,y_2)>\epsilon$ for some fixed $\epsilon>0$. 
\end{claim}
\begin{proof}
We explicitly construct a countably infinite discrete set. 
We pick any element $x_1\in X$ in the first step.
Suppose we have picked $n$ distinct points $\{x_1,x_2,\dotsc,x_n\}$ up to step $n$.
As $X$ is not totally bounded, there exists $\epsilon\in \Reals$ such that there is no finite cover of $X$ by $\epsilon$-open balls.
Let $U_i$ denote the $\epsilon$-open ball centered at $x_i$ for $i\leq n$.  
Thus, the set $X\setminus \bigcup_{i=1}^{n}U_i$ is non-empty. 
Pick any element in $X\setminus \bigcup_{i=1}^{n}U_i$ to be $x_{n+1}$. 
Note that $d(x_{n+1},x_i)>\epsilon$ for all $i\leq n$. 
Thus, we have constructed a countable discrete set $Y\subset X$ such that $d(y_1,y_2)>\epsilon$ for every pair of distinct points $y_1,y_2\in Y$. Hence, $Y$ must be closed. 
\end{proof}
Let $\cA$ be a non-principal ultrafilter and let $P$ be a finitely additive probability measure concentrated on $Y$ such that $P(A)=1$ if and only if $A\in \cA$.  
Suppose there is a sequence $\{P_n\}_{n\in \Nats}$ of countably additive probability measures converge to $P$ weakly. 
As $P$ concentrates on $Y$, with out loss of generality, we can assume $P_n$ concentrates on $Y$ for every $n\in \Nats$ (replace $P_n$ by $\frac{P_n}{P_n(Y)}$ if necessary).
\begin{claim} 
The sequence $\{P_n(B)\}_{n\in \Nats}$ converges to $P(B)$ on every subset $B\subset Y$.
\end{claim}
\begin{proof}
Fix a set $B\subset Y$.
Let $f: Y\to \Reals$ be the indicator function on $B$. 
By \cref{countdiscrete}, $f$ is a Lipschitz continuous function.
By \cref{countdiscrete} and \cref{lipextend}, $f$ can be extended to a Lipschitz continuous function from $X$ to $\Reals$.
As $\{P_n\}_{n\in \Nats}$ converges to $P$ weakly, we have $P_n(B)\to P(B)$ for every $B\subset Y$.
\end{proof}
 
As $P_1$ is countably additive, there is a finite set $B_1\subset Y$ such that $P_1(B_1)>\frac{3}{4}$. 
As $P(B_1)=0$, there exists $n_2\in \Nats$ such that $P_{n_2}(B_1)<\frac{1}{4}$. 
We can pick a finite set $B_2\subset Y$ such that $B_2\cap B_1=\emptyset$ and $P_{n_2}(B_2)>\frac{3}{4}$. 
Following this procedure, we can extract a subsequence $\{P_{n_i}: i\in \Nats\}$ from $\{P_n\}_{n\in \Nats}$ and construct a sequence of finite sets $\{B_i: i\in \Nats\}$ such that
\begin{enumerate}
\item $B_i\cap \bigcup_{j<i}B_{j}=\emptyset$ for all $i\geq 2$.
\item $P_{n_i}(B_i)>\frac{3}{4}$ for all $i\in \Nats$.
\item $P_{n_i}(\bigcup_{j<i}B_{j})<\frac{1}{4}$ for all $i\geq 2$.
\end{enumerate}
Let $A=\bigcup_{i\in \Nats}B_{2i}$. Then $P(A)$ is either $0$ or $1$ but $\{P_{n_i}(A)\}_{i\in \Nats}$ is oscillating. 
\end{proof}

The Portmanteau theorem (\cref{thmport}) gives three equivalent statements for countably additive probability measures.
It is natural to ask whether the same is true for charges. The following example shows that it is not the case.

\begin{example}\label{ezcounter}
Consider the unit interval $(0,1]$ equipped with Borel $\sigma$-algebra $\BorelSets {(0,1]}$, let $P$ be the internal push-down of the internal probability measure concentrating on $\frac{1}{2N}$ where $N\in {^{*}\Nats\setminus \Nats}$. For $n\in \Nats$, let $P_n$ be the degenerate measure at $\frac{1}{n}$. Note that the sequence $(P_n)_{n\in \Nats}$ converges weakly to $\pd{(\NSE{P})}$. As $\pd{(\NSE{P})}(\{0\})=1$, using the same proof after \cref{wccompact}, we know that $\int f \dee P_n\to \int f \dee P$ for bounded uniformly continuous $f$. On the other hand, the sequence $\int \sin{(1/x)} \dee P_n$ does not converge although $\sin{(1/x)}$ is bounded continuous on $(0,1]$. Let $A=\{\frac{1}{2n}: n\in \Nats\}$. As $P(\{0\})=0$, we know that $A$ is a continuity set. Note that $P_{n}(A)=1$ if $n$ is even and $P_{n}(A)=0$ if $n$ is odd. So the sequence $(P_{n}(A))_{n\in \Nats}$ does not converge.
\end{example}

In light of \cref{wkconverge,ezcounter}, it is natural to ask the following two questions. Suppose $P$ is a charge on a locally compact separable metric space $X$ with Borel $\sigma$-algebra $\BorelSets X$, does there exist a sequence $\{P_n\}_{n\in \Nats}$ of countably additive probability measures such that $\int f \dee P_n\to \int f\dee P$ for every bounded continuous function $f$? Does there exist a sequence $\{P_n\}_{n\in \Nats}$ of countably additive probability measures such that $P_n(A)\to P(A)$ for every $P$-continuity set $A$? These questions are answered by Miklos Laczkovich, who communicated the following example.

\begin{example}[Communicated by Miklos Laczkovich]\label{mikexample}
Let $P$ be a charge on all subsets of $\Nats$ such that $P(\{k\})=0$ for all $k\in \Nats$ and $P(\{\Nats\})=1$ (Such $P$ can be constructed via an ultrafiler on the set of natural numbers). We extend $P$ to all subsets of $\Reals$ by letting $P(A)=P(A\cap \Nats)$. We claim that there is no countable sequence $\{P_n\}_{n\in \Nats}$  of countably additive measures on $\Reals$ such that $\int f \dee P_n\to \int f \dee P$ for every bounded continuous function $f$.

\begin{proof}
Suppose there exists a sequence of $P_n$ of countably additive probability measures on $\Reals$ such that $\int f \dee P_n\to \int f \dee P$ for any bounded function $f$. We first show that if $k\geq 1$ is an integer and $0<d<1$, then $P_n([k-d,k+d])\to 0$ as $n\to \infty$. For let $d<e<1$ and let $f\geq 0$ be a continuous function which equals to 1 in $[k-d,k+d]$ and equals to $0$ outside $(k-e,k+e)$. Then $P_n([k-d,k+d])\leq \int f \dee P_n$, which converges to $\int f \dee P=0$. 

We now show that $P_{n}(\Nats)\to 1$ as $n\to \infty$. Assume that this is not valid. After passing to a subsequence, we have $P_{n}(\Nats)<1-a$ for all $n\in \Nats$, where $a>0$. Then there exists $0<d_1<1$ such that $P_{1}(U(\Nats, d_{1}))<1-\frac{a}{2}$, where $U(\Nats,d_{1})=\bigcup_{n\in \Nats}[n-d_{1},n+d_{1}]$. Let $n_1=1$. As $P_{n}([1-d_{1},1+d_{1}])$, there exists $n_2>1$ such that $P_{n_2}([1-d_{1},1+d_{1}])<\frac{a}{4}$. As $P_{n_{2}}(\Nats)<1-a$, there exists $0<d_2<d_1$ such that 
\[
P_{n_2}([1-d_{1},1+d_{1}]\cup U(\Nats, d_{2}))<1-\frac{a}{4}.
\]
As $P_{n}([1-d_{1},1+d_{1}])\to 0$ and $P_{n}([2-d_{2},2+d_{2}])\to 0$, there exists an $n_3>n_2$ such that 
\[
P_{n_3}([1-d_{1},1+d_{1}]\cup [2-d_{2},2+d_{2}])<\frac{a}{4}.
\]
Since $P_{n_3}(\Nats)<1-a$, there exists $0<d_3<d_2$ such that 
\[
P_{n_3}([1-d_{1},1+d_{1}]\cup [2-d_{2},2+d_{2}]\cup U(\Nats,d_{3}))<1-\frac{a}{4}.
\]
Continuing this procedure, we get positive numbers $d_k$ and indices $n_k$ such that $P_{n_{k}}(A)<1-\frac{a}{4}$ for all $k$, where $A=\bigcup_{k=1}^{\infty}[k-d_k,k+d_k]$. As there exists a bounded continuous function $f$ which equals to $1$ on $\Nats$ and $0$ outside of $A$, we have $\int f(x) P_{n_{k}}(\dee x)\leq P_{n_{k}}(A)<1-\frac{a}{4}$ for all $k$, contradicting with the fact that $\int f(x) P_{n_{k}}(\dee x)\to \int f(x) P(\dee x)=1$. 

Thus $P_{n}(\Nats)\to 1$ as $n\to \infty$. If $E\subset \Nats$, then $\lim\sup_{n\to \infty}P_n(E)\leq P(E)$ for all $n$. For, if $0\leq f\leq 1$ is a continuous function which equals to $1$ on $E$ and $0$ on $\Nats\setminus E$, then 
\[
P_{n}(E)\leq \int f(x) P_{n}(\dee x)\to \int f(x) P(\dee x)=P(E).
\]
This implies that $P_{n}(E)\to P(E)$. For, if $\lim\inf_{n\to \infty}P_{n}(E)<P(E)$, then from
\[
\lim\sup_{n\to \infty}P_{n}(\Nats\setminus E)\leq P(\Nats\setminus E)
\]
and from additivity it follows that $\lim\inf_{n\to \infty}P_{n}(E)<1$, which is impossible. 

Thus, we have $P(E)=\lim_{n\to \infty}P_{n}(E)$ for all $E\subset \Nats$. By Vitali-Hahn-Saks theorem, $P$ is countably additive, a contradiction.
\end{proof}

By using a similar argument, we can show that there is no 
countable sequence $\{P_n\}_{n\in \Nats}$ of countably additive probability measures on $\Reals$ such that $P_n(A)\to P(A)$ for every $P$-continuity set $A$.
\end{example}

\cref{mikexample} shows that \cref{wkconverge} is sharp. This implies that, when we talk about weak convergence to a charge, it is necessary to restrict ourselves to bounded uniformly continuous real-valued functions.

We conclude this paper by giving a nonstandard characterization of weak convergence defined in \cref{defwkconverge}.
\citet{nsweak} gave a nonstandard characterization of weak convergence of countably additive probability measures. The main result of \cite{nsweak} is the following:

\begin{theorem}[{\citep[][Thm.~4]{nsweak}}]\label{nswkconverge}
Let $X$ be a metric space equipped with Borel $\sigma$-algebra $\BorelSets X$.
Let $\{P_n\}_{n\in \Nats}$ be a sequence of countably additive probability measures on $(X,\BorelSets X)$.
Then the following are equivalent:
\begin{enumerate}
\item $\{P_n\}_{n\in \Nats}$ converges weakly to some countably additive probability measure on $(X,\BorelSets X)$.
\item For all infinite $N_1,N_2\in \NSE{\Nats}$, we have $\pd{(\NSE{P}_{N_1})}(X)=1$ and $\pd{(\NSE{P}_{N_1})}=\pd{(\NSE{P}_{N_2})}$.
\end{enumerate}
\end{theorem}

By using \cref{nswkconverge}, we present the following result.

\begin{theorem}
Let $X$ be a locally compact separable metric space equipped with Borel $\sigma$-algebra $\BorelSets X$.
Let $\{P_n\}_{n\in \Nats}$ be a sequence of countably additive probability measures on $(X,\BorelSets X)$.
Then the following are equivalent:
\begin{enumerate}
\item $\{P_n\}_{n\in \Nats}$ converges weakly to some charge on $(X,\BorelSets X)$.
\item For all infinite $N_1,N_2\in \NSE{\Nats}$, we have $\pd{(\NSE{P}_{N_1})}=\pd{(\NSE{P}_{N_2})}$.
\end{enumerate}
\end{theorem}
\begin{proof}
Suppose that $\{P_n\}_{n\in \Nats}$ converges weakly to some charge $P$ on $(X,\BorelSets X)$. Let $\comp{X}=X\cup \{a_0\}$ denote the metric one-point compactification of $X$. We extend each $P_n$ and $P$ to $\comp{X}$ by defining $P(\{a_0\})=0$ and $P_n(\{a_0\})=0$ for all $n\in \Nats$. As $\comp{X}$ is compact, by \cref{pushdown}, we know that $\pd{(\NSE{P})}$ is a countably additive probability measure on $(\comp{X},\BorelSets {\comp{X}})$.
\begin{claim}
$\{P_n\}_{n\in \Nats}$ converges weakly to $\pd{(\NSE{P})}$ in $(\comp{X},\BorelSets {\comp{X}})$.
\end{claim}
\begin{proof}
Let $f$ be a bounded continuous function from $\comp{X}$ to $\Reals$. Let $g$ be the restriction of $f$ to $X$. Then $g$ is a bounded uniformly continuous function from $X$ to $\Reals$. By assumption, we have $\int_{X} g \dee P_n\to \int_{X} g \dee P$. Thus, we have $\int_{\comp{X}}f \dee P_n\to \int_{\comp{X}} f \dee P$. By \cref{pdinteq}, we have $\int_{\comp{X}} f\dee P=\int_{\comp{X}} f \dee \pd{(\NSE{P})}$, completing the proof.
\end{proof}
By \cref{nswkconverge}, we have $\pd{(\NSE{P}_{N_1})}=\pd{(\NSE{P}_{N_2})}$ for all infinite $N_1,N_2\in \NSE{\Nats}$.

Now suppose we have $\pd{(\NSE{P}_{N_1})}=\pd{(\NSE{P}_{N_2})}$ for all infinite $N_1,N_2\in \NSE{\Nats}$. For every $n\in \NSE{\Nats}$, $\NSE{P}_n$ can be extended to an internal probability measure on $(\NSE{\comp{X}},\NSE{\BorelSets {\comp{X}}})$ by letting $\NSE{P}_n(\{a_0\})=0$. By \cref{nswkconverge}, we know that the sequence $\{P_n\}_{n\in \Nats}$ converges weakly to a countably additive probability measure $\mu$ on $(\comp{X}, \BorelSets {\comp{X}})$. Pick an element $y\in \NSE{X}$ such that $y$ is in the monad of $a_0$. Define an internal probability measure $\nu$ on $(\NSE{\comp{X}},\NSE{\BorelSets {\comp{X}}})$ as following:
\begin{enumerate}
\item $\nu(A)=\NSE{\mu}(A)$ for all $A\in \NSE{\BorelSets {\comp{X}}}$ such that $\{a_0,y\}\cap A=\emptyset$.
\item $\nu(A)=\NSE{\mu}(A)$ for all $A\in \NSE{\BorelSets {\comp{X}}}$ such that $\{a_0,y\}\subset A$.
\item $\nu(A)=\NSE{\mu}(A)-\NSE{\mu}(\{a_0\})$ for all $A\in \NSE{\BorelSets {\comp{X}}}$ such that $y\not\in A$ and $a_0\in A$.
\item $\nu(A)=\NSE{\mu}(A)+\NSE{\mu}(\{a_0\})$ for all $A\in \NSE{\BorelSets {\comp{X}}}$ such that $y\in A$ and $a_0\not\in A$.
\end{enumerate}
By the internal definition principle and the fact that $\nu(\{a_0\})=0$, $\nu$ defines an internal probability measure on $(\NSE{X},\NSE{\BorelSets X})$. Thus, the internal push-down $\ipd{\nu}$ defines a charge on $(X,\BorelSets X)$.
\begin{claim} $\pd{\nu}=\mu$. \end{claim}
\begin{proof}
Pick a set $E\in \BorelSets {\comp{X}}$. Suppose that $y\not\in \ST^{-1}(E)$. Then we know that $a_0\not\in \ST^{-1}(E)$. By the definition of $\nu$, we have $\nu(B)=\NSE{\mu}(B)$ for all $B\in \NSE{\BorelSets {\comp{X}}}$ such that $B\subset \ST^{-1}(E)$.
By \cref{stpreserve}, we have
\[
\mu(E)&=\overline{\NSE{\mu}}(\ST^{-1}(E))\\
&=\sup\{\NSE{\mu}(B): B\in \NSE{\BorelSets {\comp{X}}}\wedge B\subset \ST^{-1}(E)\}\\
&=\sup\{\nu(B): B\in \NSE{\BorelSets {\comp{X}}}\wedge B\subset \ST^{-1}(E)\}=\overline{\nu}(\ST^{-1}(E)).
\]

Now suppose that $y\in \ST^{-1}(E)$. Then we have $a_0\in E\subset \ST^{-1}(E)$. Thus, we have $\nu(B)=\NSE{\mu}(B)$ for all $B\in \NSE{\BorelSets {\comp{X}}}$ such that $B\supset \ST^{-1}(E)$. By \cref{stpreserve}, we have
\[
\mu(E)&=\overline{\NSE{\mu}}(\ST^{-1}(E))\\
&=\inf\{\NSE{\mu}(B): B\in \NSE{\BorelSets {\comp{X}}}\wedge B\supset \ST^{-1}(E)\}\\
&=\inf\{\nu(B): B\in \NSE{\BorelSets {\comp{X}}}\wedge B\supset \ST^{-1}(E)\}=\overline{\nu}(\ST^{-1}(E)).
\]
As $E$ is arbitrary, we have the desired result.
\end{proof}
Pick a bounded uniformly continuous function $f$ from $X$ to $\Reals$. We can extend $f$ to a bounded continuous function $\comp{f}$ from $\comp{X}$ to $\Reals$. By assumption, we have
\[
\int_{X} f \dee P_n=\int_{\comp{X}} \comp{f} \dee P_n\to \int_{\comp{X}} \comp{f} \dee \mu.
\]
By \cref{pdinteq}, we have
\[
\int_{\comp{X}} \comp{f} \dee \mu=\int_{\comp{X}} \comp{f} \dee \ipd{\nu}=\int_{X} f \dee \ipd{\nu}.
\]
Hence we have shown that $\{P_n\}_{n\in \Nats}$ converges weakly to $\ipd{\nu}$ which is a charge on $(X,\BorelSets X)$, completing the proof.
\end{proof}

\printbibliography


%


\end{document}